\documentclass{article}
\usepackage{amsmath,amsthm,amssymb}
\usepackage[textwidth=17cm, textheight=24cm]{geometry}

\newtheorem{theorem}{Theorem}

\begin{document}
\thispagestyle{empty}
\title{\textbf{\Large{Remark on a result of Constantine}}}

\author{\textsc{Padraig \'O Cath\'ain}\thanks{\textit{E-mail: p.ocathain@gmail.com}}\\
	Department of Mathematics and Systems Analysis and Department of Computer Science,\\
Aalto University, Finland.
}

\maketitle
\vspace{-0.2cm}
\abstract{
In this short note we construct codes of length $4n$ with $8n+8$ codewords and minimum distance $2n-2$ whenever $4n+4$ is the order of a Hadamard matrix. This generalises work of Constantine who obtained a similar result in the special case that $n$ is a prime power.}

\section{Hadamard matrices}

The Hadamard determinant bound states that an $n\times n$ matrix with entries in the complex unit circle has determinant at most $n^{n/2}$, see \cite{Hadamard1893}. It was  known to Hadamard that the order of a real-valued matrix meeting this bound with equality is $1$, $2$ or divisible by $4$. Such a ${\pm 1}$-matrix is called \textit{Hadamard}.

We refer the reader to the monographs of de Launey and Flannery, and of Horadam for comprehensive information on Hadamard matrices \cite{deLauneyFlannery,HoradamHadamard}, and to Seberry and Yamada for a survey of existence results \cite{SeberryYamadaSurvey}. In this note we will make use only of a small part of the existence theory for Hadamard matrices.

\begin{theorem}\label{HadExist}
There exists a Hadamard matrix of order $n$ if 
\begin{enumerate}
	\item $n = 2^{a}$ for any $a \geq 0$ \cite{Sylvester} 
	\item $n = q+1$ where $q \equiv 3 \mod 4$ is a prime power \cite{Paley1933}
	\item $n = 2(q+1)$ where $q \equiv 1 \mod 4$ is a prime power \cite{Paley1933} 
	\item $n = q(q+2)+1$ where $q$ and $q+2$ are both prime powers \cite{StantonSprott}
	\item $n = (s-1)^{u} + 1$ where $u$ is odd, and $s$ is the order of a skew-Hadamard matrix, in particular $s$ can be as in 1, 2 \cite{SeberryYamadaSurvey}
	\item $n = ab/2$ or $n = abcd/16$ where $a, b, c, d$ are orders of Hadamard matrices \cite{SeberryYamadaProducts, CraigenSeberryZhang}
	\item $n < 668$ \cite{Kharaghani32}
\end{enumerate}
\end{theorem}

Apart from the product theorems, all of these constructions were known to Brauer in the 1950s \cite{BrauerHadamard}. According to Brauer, the matrices normally attributed to Paley were known to Hopf and to Schur already by 1920. We did not attempt a comprehensive survey of the existence of Hadamard matrices in Theorem \ref{HadExist}, though it covers the most important known constructions. In particular, de Launey and Gordon have observed that the orders of Paley matrices and their Kronecker powers are dense in the set of known Hadamard orders \cite{deLauneyGordon}. 

\section{Codes from Hadamard matrices}

Now, let $M$ be a $n \times m$ matrix with entries in $\{\pm 1\}$. For $1 \leq i \leq n$ denote by $r_{i}$ the $i^{\textrm{th}}$ row of $M$, and by $r_{n+i}$ its negation. Define $\underline{1}$ to be the all ones vector of length $m$, and form $(0,1)$-vectors by
\[ c_{i} = \frac{1}{2} (r_{i} +\underline{1}).\]
The binary code determined by $M$ is $C_{M} = \{c_{i} \mid 1\leq i \leq 2n\}$. By linearity of the inner product, the (Hamming) distance between codewords $c_{i}$ and $c_{j}$ is \[d(c_{i}, c_{j}) = \frac{m - \langle r_{i},r_{j}\rangle}{2}.\]

As a result, the minimum distance of $C_{M}$ is equal to $\frac{m-|\alpha|}{2}$ where $\alpha$ is the off-diagonal entry of $MM^{\top}$ of largest absolute value. A $4t \times 4t$ matrix $M$ is Hadamard if and only if the minimum distance of $C_{M}$ is $2t$. See Theorem 5.1 of \cite{PetersonCoding} for an early example of this construction in coding theory.

\begin{theorem}\label{mainthm}
Suppose that there exists a Hadamard matrix of order $4t+4i$, where $i < t$. Then there exists a nonlinear code of length $4t$ with $8t+8i$ codewords and minimum distance $2t-2i$.
\end{theorem}

\begin{proof}
Form a $(4t+4i) \times 4t$ matrix $M$ by deleting any $4i$ columns of $H$. Observe that the inner product of any two rows of $M$ differs from the inner product of the corresponding rows of $H$ by at most $4i$. In particular the largest off-diagonal entry of $MM^{\top}$ is at most $4i$, and the minimum distance of $C_{M}$ is at least $\frac{4t-4i}{2} = 2t-2i$, as required.
\end{proof}

For odd prime powers $q$, Constantine has obtained a closely related construction which yields a non-linear code of length $4q$ with $8q$ codewords and minimum distance $2q-2$, using a plug-in construction and properties of quadratic residues \cite{Constantine}. He proposes that such codes might be a useful substitute for Hadamard codes at orders where no Hadamard matrix is known to exist. 

Theorem \ref{mainthm} strengthens Constantine's result in two ways. For any odd prime power $q$, there exists a Hadamard matrix of order $4q+4$ by Theorem \ref{HadExist}: it suffices to consider Kronecker products of Paley and of Sylvester matrices. Applying Theorem \ref{mainthm} with $i = 1$ to such a matrix gives a code of the same length and minimum distance as Constantine's, but with an additional eight codewords. Theorem \ref{HadExist} also gives many additional Hadamard orders to which Theorem \ref{mainthm} can be applied.

\section{Comments on Maximality and Optimality}

Recall that a binary code $C$ is self-complementary if $c_{i} + \underline{1} \in C$ for all $c_{i} \in C$. In this section we will consider only self-complementary codes. We say that a code $C$ is \textit{maximal} if including any additional codeword necessarily decreases the minimum distance, and that $C$ is \textit{optimal} if $|C|$ is maximal among all codes with the same length and minimal distance. In this section we show that for any fixed $i$, all sufficiently large codes obtained from Theorem \ref{mainthm} are maximal, and are within a constant number of codewords of being optimal.

\begin{theorem} 
If $t > 16i^{2} - i$ then any code obtained from Theorem \ref{mainthm} is maximal. 
\end{theorem}

\begin{proof}
Set $m = 4t+4i$ and let $H$ be Hadamard of order $m$. Recall that the rows of $m^{-\frac{1}{2}}H$ form an orthonormal basis of $\mathbb{R}^{m}$; so for any vector $v$ we have the identity 
\[ \|v\|_{2}^{2} = \sum_{i=1}^{m} \| \langle m^{-\frac{1}{2}} r_{i}, v \rangle \|_{2}^{2} = m^{-1} \sum_{i=1}^{m} \| \langle r_{i}, v \rangle \|_{2}^{2}. \]
In particular, if $v$ is a $\pm 1$ vector, then $\|v\|_{2}^{2} = m$, and we have that $ \sum_{i=1}^{m} \| \langle r_{i}, v \rangle \|_{2}^{2} = m^{2}$. Hence, there exists an index $j$ such that either $\langle r_{j}, v \rangle  \geq \sqrt{m}$ or $\langle -r_{j}, v \rangle  \geq \sqrt{m}$. 

Now denote by $c_{v}$ the binary codeword obtained from $v$. Then one of $d(c_{j}, c_{v})$ and $d(c_{m+j}, c_{v})$ is less than $\frac{m-\sqrt{m}}{2} < 2t-2i$. This holds for any vector $v$, so $C_{H}$ is maximal, as claimed.
\end{proof}

For any self-complementary binary code $C$ with length $n$ and minimum distance $d$, the Grey-Rankin bound states that 
\[ |C| \leq \frac{8d(n-d)}{n-(n-2d)^{2}}, \] 
whenever the right-hand side is positive (see \cite{McGuire} for a simple proof). It is straightforward to see that a code with $(n,d) = (4t, 2t)$ has at most $8t$ codewords, and hence that there exists a code with these parameters meeting the Grey-Rankin bound if and only if there exists a Hadamard matrix of order $4t$. More generally, McGuire has given a characterisation of Grey-Rankin optimal codes in terms of quasi-symmetric designs and Hadamard matrices \cite{McGuire}. 

Optimal codes with $(n,d) = (4t-1, 2t-1)$ are obtained by puncturing a Hadamard code in one co-ordinate. One may consider codes punctured twice or three times: it is an easy exercise to see that (for codes of length $>160$) the codes obtained from a Hadamard matrix are respectively within $4$ and within $12$ codewords of the Grey-Rankin bound. The codes of Theorem \ref{mainthm} are obtained by puncturing a Hadamard code $4i$ times. An easy computation shows that in this case the Grey-Rankin bound specialises to 
\[ |C| \leq 8t + 32i^{2} + \frac{128i^{4} - 8i^{2}}{t-4i^{2}}. \] 

For $i=1$, the Grey-Rankin bound is integral only when $t-4$ is a divisor of $120$. McGuire investigated existence of codes meeting the bound for certain small parameter sets, and found that for $t = 5, 7, 9$ codes meeting the bound exist, but not for $t = 6$, see \cite{McGuire}. Clearly when $t > 124$ (i.e. for lengths $>476$) the codes of Theorem \ref{mainthm} are within $24$ codewords of optimality.

More generally, for any fixed $i>0$ and for all sufficiently large $t$, we have that $|C| < 8t + 32i^{2}$. Thus the codes obtained from Theorem \ref{mainthm} are within a constant number of codewords of optimality. Few examples of codes meeting the Grey-Rankin bound are known when $i>0$. McGuire has described a family of codes coming from the action of the symplectic group $\textrm{Sp}_{2n}(2)$ on elliptic hyperovals which have parameters $(n,d) = (2^{2l-1} - 2^{l-1}, 2^{2l-2}-2^{l-1})$ for $l>0$. These symplectic codes are optimal and have size $2^{2l+1}$. Solving for the parameters of Theorem \ref{mainthm}, we find that $4t+4i = 2^{2l-1}$. Such Hadamard matrices exist by Theorem \ref{HadExist}, and we obtain a code with $2^{2l}$ codewords. So even when $i$ grows exponentially, the codes of Theorem \ref{mainthm} remain within a factor of $2$ of optimality.

It remains an open problem to construct, for any fixed $i>0$, an infinite family of codes with parameters $(n,d) = (4t, 2t-2i)$ where $|C|> 8t+8i$. We consider the case $i = 1$ to be particularly interesting. 

\section*{Acknowledgements}

This research was partially supported by the Academy of Finland (grants \#276031, \#282938, and \#283262 and \#283437). The support from the European Science Foundation under the COST Action IC1104 is also gratefully acknowledged. 

\bibliographystyle{abbrv}
\flushleft{
\bibliography{NewBiblio}

\def\Dbar{\leavevmode\lower.6ex\hbox to 0pt{\hskip-.23ex \accent"16\hss}D}
\begin{thebibliography}{10}

\bibitem{BrauerHadamard}
A.~Brauer.
\newblock On a new class of {H}adamard determinants.
\newblock {\em Math. Z.}, 58:219--225, 1953.

\bibitem{Constantine}
G.~M. Constantine.
\newblock An alternative to {H}adamard codes - one error the price for
  existence.
\newblock {\em Annals of Combinatorics}, 19(3):421--425, 2015.

\bibitem{CraigenSeberryZhang}
R.~Craigen, J.~Seberry, and X.~M. Zhang.
\newblock Product of four {H}adamard matrices.
\newblock {\em J. Combin. Theory Ser. A}, 59(2):318--320, 1992.

\bibitem{deLauneyFlannery}
W.~{d}e Launey and D.~Flannery.
\newblock {\em Algebraic design theory}.
\newblock Mathematical Surveys and Monographs, vol. 175. American Mathematical
  Society, Providence, RI, 2011.

\bibitem{deLauneyGordon}
W.~de~Launey and D.~M. Gordon.
\newblock On the density of the set of known {H}adamard orders.
\newblock {\em Cryptogr. Commun.}, 2(2):233--246, 2010.

\bibitem{Hadamard1893}
J.~Hadamard.
\newblock R\'{e}solution d'une question relative aux d\'{e}terminants.
\newblock {\em Bull. Sci. Math.}, 17:240--246, 1893.

\bibitem{HoradamHadamard}
K.~J. Horadam.
\newblock {\em Hadamard matrices and their applications}.
\newblock Princeton University Press, Princeton, NJ, 2007.

\bibitem{Kharaghani32}
H.~Kharaghani and B.~Tayfeh-Rezaie.
\newblock On the classification of {H}adamard matrices of order 32.
\newblock {\em J. Combin. Des.}, 18(5):328--336, 2010.

\bibitem{McGuire}
G.~McGuire.
\newblock Quasi-symmetric designs and codes meeting the grey?rankin bound.
\newblock {\em Journal of Combinatorial Theory, Series A}, 78(2):280 -- 291,
  1997.

\bibitem{Paley1933}
R.~Paley.
\newblock On orthogonal matrices.
\newblock {\em J. Math. Phys.}, 12:311--320, 1933.

\bibitem{PetersonCoding}
W.~W. Peterson and E.~J. Weldon, Jr.
\newblock {\em Error-correcting codes}.
\newblock The M.I.T. Press, Cambridge, Mass.-London, second edition, 1972.

\bibitem{SeberryYamadaProducts}
J.~Seberry and M.~Yamada.
\newblock On the products of {H}adamard matrices, {W}illiamson matrices and
  other orthogonal matrices using {$M$}-structures.
\newblock {\em J. Combin. Math. Combin. Comput.}, 7:97--137, 1990.

\bibitem{SeberryYamadaSurvey}
J.~Seberry and M.~Yamada.
\newblock Hadamard matrices, sequences, and block designs.
\newblock In {\em Contemporary design theory}, Wiley-Intersci. Ser. Discrete
  Math. Optim., pages 431--560. Wiley, New York, 1992.

\bibitem{StantonSprott}
R.~G. Stanton and D.~A. Sprott.
\newblock A family of difference sets.
\newblock {\em Canad. J. Math.}, 10:73--77, 1958.

\bibitem{Sylvester}
J.~Sylvester.
\newblock Thoughts on inverse orthogonal matrices, simultaneous sign
  successions, and tessellated pavements in two or more colours, with
  applications to newton's rule, ornamental tile-work, and the theory of
  numbers.
\newblock {\em Phil. Mag.}, 34(1):461--475, 1867.

\end{thebibliography}
}

\end{document}